\documentclass{gtart}

\usepackage[inner=20mm, outer=20mm, textheight=245mm]{geometry}

\usepackage{amsmath, amssymb, latexsym, amscd, graphicx, epstopdf}
\usepackage{euscript, courier}
\usepackage{hyperref, url}
\usepackage{subfigure}
\usepackage{verbatim}
\usepackage[makeroom]{cancel}
\usepackage{caption}
\usepackage{mathabx}

\def\CC{\mathbb{C}}

\def\PP{\mathbb{P}}
\def\QQ{\mathbb{Q}}
\def\RR{\mathbb{R}}
\def\ZZ{\mathbb{Z}}

\def\RP2{\RR \PP^2}

\newcommand{\inv}[1]{{#1}^{-1}}

\DeclareMathOperator{\SL}{SL}
\DeclareMathOperator{\PSL}{PSL}

\def\X{\mathfrak{X}}

\def\R{\mathfrak{R}}

\def\P{\mathcal{P}}

\def\T3p{\mathcal{T}_3^+}

\newcommand{\wt}[1]{\widetilde{#1}}
\newcommand{\aut}{\mathrm{Aut}}

\def\bound{\partial}

\DeclareMathOperator{\Hom}{Hom}
\DeclareMathOperator{\tr}{tr}

\newcommand\restr[2]{{
		\left.\kern-\nulldelimiterspace 
		#1 
		\vphantom{\big|} 
		\right|_{#2} 
}}

\theoremstyle{plain}

\newtheorem{thm}{Theorem}
\newtheorem*{thm*}{Theorem}

\newtheorem*{lem*}{Lemma}
\newtheorem{cor}[thm]{Corollary}
\newtheorem*{cor*}{Corollary}

\newtheorem*{cla*}{Claim}
\newtheorem{pro}[thm]{Proposition}
\newtheorem*{pro*}{Proposition}

\newtheorem*{rem*}{Remark}

\newtheorem*{defn*}{Definition}

\newtheorem*{defns*}{Definitions}
\newtheorem*{exa*}{Example}

\newtheorem*{exe*}{Exercise}

\newtheorem*{que*}{Question}

\addtocounter{MaxMatrixCols}{4}

\begin{document}
\title{Ideal points of character varieties, algebraic non-integral\\ representations, and undetected closed essential surfaces in 3--manifolds} 
\author{Alex Casella, Charles Katerba and Stephan Tillmann}
	
\begin{abstract}
Closed essential surfaces in a 3--manifold can be detected by ideal points of the character variety or by algebraic non-integral representations. We give examples of closed essential surfaces not detected in either of these ways.
For ideal points, we use Chesebro's module-theoretic interpretation of Culler-Shalen theory. 
As a corollary, we construct an infinite family of closed hyperbolic Haken 3--manifolds with no algebraic non-integral representation into $\PSL_2 ( \CC)$, resolving a question of Schanuel and Zhang. 
\end{abstract}
	
	

	
\maketitle
	
	\section{Introduction}

A \emph{knot manifold} is a compact, irreducible, orientable 3--manifold whose boundary consists of a single torus.
We say that a knot manifold is \emph{large} if it contains a closed essential surface. Here, an \emph{essential surface} $\Sigma$ in a 3--manifold is an orientable, properly embedded surface with no sphere or boundary parallel components such that the homomorphism on fundamental groups induced by inclusion is injective for each connected component of $\Sigma$. 

Let $N$ be a knot manifold and 
$\Gamma$ denote its fundamental group.
Culler and Shalen \cite{CS1} construct essential surfaces in 3--manifolds from representations of
$\Gamma$ into $\SL_2 ( \CC)$. This combines algebraic geometry, valuations and actions on trees, and 
has seen broad applications in 3--manifold topology (see for example \cite{BZ2, BZ1, CCGLS, CGLS}).  In this introduction, we assume some familiarity with Culler-Shalen theory---basic definitions and facts are collected in \S\ref{sec:charvar} and \S\ref{sec:preliminaries}.

The set $\X_{\SL} ( N )$ of characters of representations of $\Gamma$ into $\SL_2 ( \CC)$ admits the structure of a complex affine algebraic set called the {character variety of $N$}.  Essential surfaces can be associated to certain representations $\Gamma \to \SL_2 ( F)$, where $F$ is a field with a valuation $v \co F^\times \to \ZZ.$ The character variety provides two ways to find such representations: by passing to ideal points and by carrying algebraic non-integral, or ANI, representations. If the essential surface $\Sigma$ in $N$ is associated to an ideal point of a curve in $\X_{\SL}(N)$ it is \emph{detected by (an ideal point of) the character variety}.  Similarly, if $\Sigma$ is associated to an ANI-representation, we say $\Sigma$ is \emph{ANI-detected by the character variety}. Similar terminology is adopted for the \emph{boundary slopes} of essential surfaces; that is, unoriented isotopy classes of simple closed curves on the boundary of $N$ that can be represented by the boundary components of essential surfaces in $N$. The theory and its applications were extended to representations into $\PSL_2 ( \CC)$ by Boyer and Zhang~\cite{BZ2}.

This paper addresses the general question of which essential surfaces in a 3--manifold are detected by ideal points or ANI-representations. To this end, Chesebro and the third author~\cite{CT} showed that there are boundary slopes of knot manifolds which are not detected by the character variety. Motegi showed that there are closed graph manifolds that contain essential tori \cite{MOT} not detected by the character variety. 
Boyer and Zhang~\cite[Theorem 1.8]{BZ2} showed that there are infinitely many closed hyperbolic 3--manifolds whose character varieties do not detect closed essential surfaces contained in these manifolds. 
Schanuel and Zhang~\cite[Example 17]{SZ} gave an example of a closed hyperbolic 3--manifold with a closed essential surface that cannot be detected by an ideal point but is ANI-detected. We first turn our attention to knot manifolds, and as an application answer an open question of Schanuel and Zhang. The knots $10_{152}$, $10_{153}$ and $10_{154}$ were shown to be large by Burton, Coward and the third author~\cite{BCT}. This paper shows:

\begin{thm}
Let $N$ be the complement in $S^3$ of the large hyperbolic knot $10_{152}$, $10_{153}$ or $10_{154}$.  
No closed essential surface in $N$ is detected by an ideal point of the character variety of $N$.
\label{thm:main}
\end{thm} 

The proof of Theorem \ref{thm:main} uses a module-theoretic approach to Culler-Shalen theory developed by Chesebro~\cite{C}.  As will be explained further in \S\ref{sec:main}, this approach transforms the problem into a computation in commutative algebra.  Thus, we are able to answer this question using algorithmic techniques implemented in the software package \textit{Macaulay2}~\cite{M2}.  It also follows from \cite[Proposition 5.2]{C} that no closed essential surface in the complement of $10_{152}$, $10_{153}$ or $10_{154}$ is ANI-detected.

Joshua Howie informed the authors that $10_{152}$, $10_{153}$, and $10_{154}$ are the first non-alternating adequate knots in Rolfsen's table.  It would be interesting to know the extent to which Howie's observation may give a geometric or topological obstruction to detecting closed essential surfaces by ideal points. 

As a corollary to Theorem \ref{thm:main}, we answer Question 9 of Schanuel and Zhang~\cite{SZ}.  They ask whether there are large closed hyperbolic 3--manifolds which have no ANI-representations into $\PSL_2 ( \CC).$ We answer this question affirmatively. 

\begin{cor} 
There are infinitely many large closed hyperbolic  3--manifolds with no ANI-representations into $\PSL_2(\CC)$.
\label{cor:noani}  
\end{cor}

Corollary~\ref{cor:noani} follows from Theorem~\ref{thm:main} by considering sufficiently large Dehn fillings of any of the knots given by the theorem.  Under such a Dehn filling, a closed essential surface in the knot exterior will remain essential in the filled manifold \cite[Theorem 2.0.3]{CGLS} and the filled manifold will be hyperbolic by Thurston's Hyperbolic Dehn Surgery Theorem.  Using a result of Culler about lifting representations and a result of Chesebro connecting ANI-detected to detected closed essential surfaces, we show that the filled manifold cannot have any ANI-representations into $\PSL_2( \CC)$. 

The remainder of the paper is structured as follows.  In the next section, we review some basics concerning character varieties and describe an algorithm for their computation.  Section \ref{sec:preliminaries} outlines the essentials of Culler-Shalen theory and summarizes Chesebro's module-theoretic approach.  Section \ref{sec:main} begins with a description of our computational techniques and heuristic for finding knots whose character varieties may not detect closed essential surfaces.  We then prove the main results of this paper.

\textbf{Acknowledgements}\ The first author acknowledges support by the Commonwealth of Australia.
The second author acknowledges support by an NSF-EAPSI Fellowship (project number 1713920).
Research of the third author was supported by an Australian Research Council Future Fellowship (project number FT170100316). The authors thank Robert Loewe and Benjamin Lorenz for running the computation for $10_{154}$ on a cluster at TU Berlin.

\section{Character varieties}
\label{sec:charvar}

We describe the construction of character varieties  in the $\SL_2( \CC )$ case. The $\PSL_2 ( \CC)$ case is similar with only a few additional technicalities; see \cite{BZ2} for a detailed account of Culler-Shalen theory for $\PSL_2 ( \CC )$.

Let $\Gamma$ be a finitely presented group with presentation $\langle \gamma_1, \cdots, \gamma_n \mid r_1 , \ldots, r_m  \rangle$. A function $\rho \co \{\gamma_{j} \}_{j=1}^n \to \SL_2 (\CC)$ extends to a representation if and only if $\rho ( r_i)$ is the $2 \times 2$ identity matrix for each $1 \leq i \leq m$. The set $\R_{\SL}(\Gamma) := \Hom ( \Gamma, \SL_2 ( \CC))$ is therefore in natural 1-1 correspondence with an algebraic set in $\CC^{4n}$ and hence called the \emph{representation variety of $\Gamma$}. 

For each $\gamma \in \Gamma$ there is a regular function $I_\gamma \co \R_{\SL}( \Gamma ) \to \CC$ given by $I_\gamma ( \rho ) = \tr \rho ( \gamma )$.  Let $T( \Gamma)$ denote the ring with unity generated the set $\{I_\gamma \}_{\gamma \in \Gamma}$.  

\begin{pro}
{(cf. \cite{CS1} and \cite{GMA})} The ring $T(\Gamma)$ is generated by the set
 \[ \mathcal{G} = \{ I_{\gamma_{i_1} \cdots \gamma_{i_k}}   \mid \text{for } 1 \leq i_1 < \cdots <  i_k \leq n \text{ and }  k \leq 3 \}. \]
\label{prop:tgfg}
\end{pro}

\begin{proof}
We include a sketch of a proof of this proposition to remind the reader that there is an algorithm to write each $I_\gamma$ for $\gamma \in \Gamma$ as a polynomial in the elements of $\mathcal{G}$.  This algorithm is based on the following trace identities \cite[Lemmas 4.1 and 4.1.1]{GMA}.  If $A, B, C, D \in \SL_2 ( \CC)$, then 

\begin{align}
\tr A & =  \tr A^{-1} \label{eq:1}  \\
\tr AB & =   \tr A \tr B - \tr A \inv{B} \\
\tr ACB &  =  \tr A \tr BC + \tr B \tr AC + \tr C \tr AB - \tr A \tr B \tr C - \tr ABC  \\
 2 \cdot \tr ABCD & =  \tr A \tr BCD + \tr B \tr ACD + \tr C \tr ABD - \tr D \tr ACB  + \tr BC \tr AD \nonumber \\ 
 &  \qquad  + \tr AB \tr CD - \tr AC \tr BD + \tr B \tr D \tr AC - \tr A \tr B \tr CD \nonumber \\
 &  \qquad - \tr B \tr C \tr AD 
 \end{align}
Given $\gamma \in \Gamma$, write $\gamma$ as a word in the generators $\{\gamma_i\}_{i = 1}^n$.  Use the first and second trace identities to write $I_\gamma$ as a polynomial in trace functions of words with no inverses or exponents on letters higher than one. With the fourth trace identity one can reduce the word length to at most 3. Finally, the third identity allows one to write the trace functions in lexicographic order.
\end{proof} 

Order the $N = n(n^2 +5)/6$ elements of $\mathcal{G}$ lexicographically and use this ordering to define a map $t \co \R_{\SL} ( \Gamma ) \to \CC^N$ by $t( \rho ) = ( I_g ( \rho ) )_{ g \in \mathcal{G}}$.  Culler and Shalen proved that the image of $t$ is a closed algebraic set that is  in 1-1 correspondence with the set of  $\SL_2( \CC)$-characters of $\Gamma$ \cite{CS1}. Thus, $\X_{\SL}(\Gamma) := t ( \R_{\SL}( \Gamma))$ is called the \emph{character variety of $\Gamma$}. The elements of $\mathcal{G}$ serve as coordinate functions on $\X_{\SL}( \Gamma)$. 

Gonz\'alez-Acu\~na and Montesinos-Amilibia~\cite{GMA} exhibited specific defining equations for the character variety of a finitely presented group.   First, they exhibit a finite collection of polynomials which cut out the character variety for the free group $F_n$ on $n$ letters.  Next, they prove that $\X_{\SL} ( \Gamma )$ is the algebraic  subset of $\X_{\SL} ( F_n )$ cut out by the $m(n+1)$ polynomials 
\begin{align*}
 \{ I_{r_1} - 2, I_{ \gamma_1 r_1} - I_{\gamma_1}, \ldots,  I_{ \gamma_n r_1} - I_{\gamma_n}, I_{r_2} -2, \ldots, I_{\gamma_n r_2} - I_{\gamma_n}, \ldots, I_{\gamma_n r_m } - I_{\gamma_n}  \}. \end{align*}
Each of the above trace functions can be computed algorithmically, as can defining equations for the character variety of the free group on $n$ letters.  By the fourth trace identity, we see that $\X_{\SL}(\Gamma)$ is cut out of $\CC^N$ by polynomials with rational coefficients.  This gives the following proposition: 

\begin{pro}
Given a presentation $\langle \gamma_1, \ldots, \gamma_n | r_1, \ldots, r_m \rangle $ for a group $\Gamma$,  defining equations for $\X_{\SL}( \Gamma)$ can be computed algorithmically. Moreover, $\X_{\SL} ( \Gamma)$ is defined over $\QQ$.
\label{prop:dfneqns}
\end{pro}

\section{Elements of Culler-Shalen theory}
\label{sec:preliminaries}

Fix a knot manifold $N$, let $\Gamma$ denote the fundamental group of $N$, and  set $\X_{\SL} ( N ) = \X_{\SL} (\Gamma) $.

If $F$ is a field with a valuation $v \co F^\times \to \ZZ$, then Bass-Serre theory gives a simplicial tree $T_v$ associated to $v$ and an action of $\SL_2 ( F)$ on $T$ \cite{JPS}.  A representation $\rho \co \Gamma \to \SL_2 ( F)$ induces an action of $\Gamma$ on $T_v$. When this action is nontrivial (that is, when no vertex of $T_v$ is fixed by $\Gamma$) a construction due to Stallings gives essential surfaces in $N$.  Essential surfaces that can be built from the above procedure are \emph{associated to the tree $T_v$}.  The next theorem demonstrates a connection between the topology of $N$, the representation $\rho$, and the valuation $v$.

\begin{thm}{\cite{CS1}}
Suppose there is a representation $\rho \co \Gamma \to \SL_2 ( F ) $ where $F$ is a field with a valuation $v$ such that the induced action of $\Gamma$ on the tree $T_v$ is nontrivial.  
\begin{enumerate}
\item If $v ( \tr  \rho ( \gamma) )  \geq 0$ for each $\gamma \in \pi_1 ( \bound N)$, then there is a closed essential surface associated to $T_v$.
\item Otherwise there is a unique element $\gamma \in \pi_1 ( \bound N) $ (up to inversion and conjugation) such that $v ( \tr \rho ( \gamma )) \geq 0$.  In this case, every essential surface associated to $T_v$ has non-empty boundary and $\gamma$ represents the boundary slope of these surfaces. 
\end{enumerate}
\label{thm:cs1}
\end{thm}

The character variety $\X_{\SL} ( N )$ provides two ways of finding fields equipped with valuations and hence essential surfaces in $N$:  by passing to ideal points and by carrying algebraic non-integral representations.  

\noindent \emph{Ideal points:} The dimension of $\X_{\SL}(N)$ is at least one \cite{CCGLS}, so take an irreducible curve  $X \subset \X_{\SL}( N )$ with normalization $\phi \co \overline{X} \to X$ and a smooth projective model $\wt{X}$ for $X$.  Then there is a birational isomorphism $\iota \co \wt{X} \to \overline{X}$ whose inverse is defined on all of $\overline{X}$.   The elements of the set 
\[ \wt{X} - \iota^{-1} ( \overline{X} )  \] 
are the \emph{ideal points of $X$}.  Up to isomorphism $X$ has a unique smooth projective model, so the set of ideal points is well-defined.  Each rational function on $X$ extends to a function $\wt{X} \to \CC P^1 = \CC \cup \{\infty \}$ and to each ideal point $\hat{x}$ of $X$ there is a natural valuation $v_{\hat{x}}$ on the function field $\CC(X) = \CC( \wt{X})$ given by 
\[ v_{\hat{x}} (f) = \begin{cases}
-(\text{order of the pole of $f$ at $\hat{x}$}) & \text{if } f(\hat{x}) = \infty \\
0 &\text{it } f(\hat{x}) \in \CC - \{ 0 \} \\ 
\text{order of the zero of $f$ at $\hat{x}$} & \text{if } f(\hat{x}) = 0  \end{cases}. \]

There is a representation, often referred to as the \emph{tautological representation}, $\P \co \Gamma \to \SL_2 ( \CC ( \wt{X}))$.  That $X$ is a curve implies that the action of $\Gamma$ on the tree $T_{v_{\hat{x}}}$ for each ideal point $\hat{x}$ of $X$ is nontrivial. Hence Theorem \ref{thm:cs1} applies. 

\noindent \emph{ANI-representations:} Suppose $F$ is a number field (i.e.\thinspace a finite extension of $\QQ$) and $\rho \co \Gamma \to \SL_2 ( F)$ is a representation.  We say $\rho$ is \emph{algebraic non-integral}, or \emph{ANI} if there is some element $\gamma \in \Gamma$ such that $\tr \rho ( \gamma)$ is not integral over $\ZZ$.  Recall that the integral closure of $\ZZ$ in $F$ is the intersection of all the valuation rings of $F$.  Since $\tr  \rho ( \gamma)$ is not integral over $\ZZ$, there is some valuation $v$ on $F$ such that $v(\tr \rho (\gamma)) < 0$.  Hence, there is an essential surface in $N$ associated to the tree $T_v$ by Theorem \ref{thm:cs1}

 Chesebro~\cite{C} noticed a connection between the  detection of essential surfaces by ideal points and an infinite collection of modules associated to the coordinate ring $\CC[X]$ of $X$ which we now describe. 

For any unital subring $R$ of $\CC$, let $T_R ( X)$ denote the $R$-subalgebra of $\CC[X]$ generated by $\mathcal{G}$, noting that $T_\CC ( X) = \CC[X]$.  If $\gamma \in \pi_1 ( N )$, then $I_\gamma$ is an element of $T_{\ZZ}(X)$ and hence $R[I_\gamma] \subseteq T_R(X)$ for each unital subring $R\subseteq \CC$.  In particular, we may view $T_R(X)$ as a $R[I_\gamma]$-module.  The following theorem relates the detection of essential surfaces by ideal points to these modules. 

\begin{thm}{(Chesebro~\cite[Theorem 1.2]{C})}
Let $X$ be an irreducible component of $\X_{\SL}(N)$ and take $R \in \{ \ZZ, \QQ, \CC \}$. Then 
\begin{enumerate}
\item $T_R(X)$ is not finitely generated as an $R[I_\gamma]$-module for each $\gamma \in \pi_1 ( \bound N)$ if and only if $X$ detects a closed essential surface.   
\item Otherwise, $T_R(X)$ is not finitely generated as a $R[I_\gamma]$-module for some $\gamma \in \pi_1 (\bound N)$ if and only if $\gamma$ represents a boundary slope detected by $X$. 
\end{enumerate}
\label{thm:cheese}
\end{thm}

	\section{Main Results}
	\label{sec:main}

We aim to prove that there are knots in $S^3$ whose complement contains closed essential surfaces, none of which are detected by the character variety.  A corollary to Theorem~\ref{thm:cheese} which replaces $X$ with an arbitrary union of irreducible components of $\X_{\SL} ( N)$ will be our main tool.

\begin{cor}
Suppose $Y = X_1 \cup \cdots \cup X_n$ is a union of irreducible components of $\X_{\SL}(N)$.  Then $Y$ does not detect a closed essential surface if and only if $\CC[Y]$ is finitely generated as a $\CC[I_\gamma]$-module for some $\gamma \in \pi_1 ( \bound N)$ which is not a boundary slope of $N$. 

Moreover, when $\CC[Y]$ is finitely generated, it is a free $\CC[I_\gamma]$-module. 
\label{cor:union}
\end{cor} 

\begin{proof}
Take an element $\gamma \in \pi_1 ( \bound N)$ which does not represent a boundary slope of $N$.  If $Y$ does not detect a closed essential surface, then, for each $i$, $\CC[X_i]$ is a finitely generated $\CC[I_\gamma]$-module by Theorem~\ref{thm:cheese}.  Since $X_i$ is irreducible, $\CC[X_i]$ is an integral domain, and so $\CC[X_i]$ is torsion free over $\CC[I_\gamma]$.  In particular, $\CC[X_i]$ is free over $\CC[I_\gamma]$ since $\CC[I_\gamma]$ is a principle ideal domain.  

Suppose $\X_{\SL}(N) \subseteq \CC^k$.  Let $\Phi \co \CC[z_1, \ldots, z_k] \to \CC[Y]$ and $\phi_i \co \CC[z_1, \ldots, z_k] \to \CC[X_i]$ denote the natural epimorphisms induced by the inclusions $Y \hookrightarrow \CC^k$ and $X_i \hookrightarrow \CC^k$.  Define a function 
 $\Psi \co \CC[Y] \to \oplus_1^n \CC[X_i]$ by 
 \[ \Psi ( \Phi( f )) = ( \phi_1 ( f ) , \ldots, \phi_n ( f )) \quad \text{for} \quad f\in \CC[z_1, \ldots, z_k]. \]
 If we regard $\oplus_1^n \CC[X_i]$ as a $\CC[I_\gamma]$-module with the diagonal action, then $\Psi$ is $\CC[I_\gamma]$-linear. 
 
 We claim that $\Psi$ is injective.  Take a polynomial $f \in \CC[z_1, \ldots, z_n]$ such that $\phi_i ( f) $ is the zero function on $X_i$ for each $i$.  Then $f$ is an element of the ideal $I(X_i)$ of $X_i$ and hence 
 \[  f \in \bigcap_1^n I(X_i) = I(Y). \]
Thus $\Psi$ is a monomorphism and $\CC[Y]$ is isomorphic to a submodule of a finitely generated free module.  Finally, since $\CC[I_\gamma]$ is a PID, $\CC[Y]$ must be a finitely generated free $\CC[I_\gamma]$-module.  
 
 Now suppose toward a contradiction that $\CC[Y]$ is finitely generated over $\CC[I_\gamma]$ and that $Y$ detects a closed essential surface.  Then, by Theorem \ref{thm:cheese}, $\CC[X_i]$ is not finitely generated over $\CC[I_\gamma]$ for some $1 \leq i \leq n$.  But the inclusion $X_i \hookrightarrow Y$ induces a surjection $\CC[Y] \to \CC[X_i]$.  This gives a contradiction since the image of a generating set for $\CC[Y]$ would generate $\CC[X_i]$.    
\end{proof}

Corollary \ref{cor:union} demonstrates how we transform the question of whether or not $\X_{\SL}(N)$ detects a closed essential surface into a computational commutative algebra problem.  The program \textit{Macaulay2} uses inexact numbers when working over $\CC$, so we must extend Corollary \ref{cor:union} so that the coefficient field is $\QQ$ where the program performs exact calculations. The proof of the following corollary is essentially \cite[Proposition 16]{K}.    

\begin{cor} 
Take $Y = X_1 \cup \cdots \cup X_n$ to be a union of irreducible components of $\X_{\SL}(N)$ and fix $\gamma \in \pi_1 ( \bound N)$.  Suppose  $Y$ is defined over $\QQ$. 
Then $T_\QQ ( Y)$  is finitely generated over $\QQ[I_\gamma]$ if and only if $\CC[Y]$ is finitely generated over $\CC[I_\gamma]$. 

In particular $Y$ does not detect a closed essential surface if and only if $T_\QQ (Y)$ is finitely generated over $\QQ[I_\gamma]$ for some $\gamma \in \pi_1 ( \bound N)$ that does not represent a boundary slope of $N$.
\label{cor:overqq}
\end{cor}

\begin{proof}
First, observe that any generating set for $T_\QQ ( Y )$ over $\QQ[ I_\gamma]$ will automatically generate $\CC[Y]$ over $\CC[I_\gamma]$.  

Now suppose $\CC[Y]$ is finitely generated over $\CC[I_\gamma]$.  Then $\CC[Y]$ is a free module and we may take a free basis $\mathcal{B} = \{b_1, \ldots, b_m\}$ lying in $T(X)$ (see the remarks following  Corollary 2.5 in \cite{C}).  

We claim $\mathcal{B}$ spans $T_\QQ ( X)$ over $\QQ[I_\gamma]$. Take an element $f \in T_\QQ(X)$ and write  $ f = \sum_1^m p_j b_j $ for some $p_j \in \CC[I_\gamma]$.  The field automorphism group $\aut( \CC / \QQ )$ acts on $\CC[Y]$ since $Y$ is defined over $\QQ$.  Moreover, since $f \in T_\QQ (X)$, if $\sigma \in \aut ( \CC / \QQ)$, 
\[ \sigma\cdot f = f, \quad \text{so} \quad  \sum_1^m (p_j - \sigma \cdot p_j) b_j = 0. \] 
$\mathcal{B}$ is a free basis, so $p_j = \sigma \cdot p_j$ for every $\sigma \in \aut ( \CC / \QQ)$.  The fixed field of $\aut( \CC / \QQ )$ is $\QQ$ \cite[Theorem 9.29]{MFT}, so each $p_j \in \QQ[I_\gamma]$.  
\end{proof}

Corollaries \ref{cor:union} and \ref{cor:overqq} provide a way to prove our main result, Theorem \ref{thm:main}.  For a given knot manifold $N$, to determine whether $\X_{\SL}(N)$ detects any closed essential surface, one must first find a slope $\alpha$ which is not a boundary slope of $N$, then decide if $T_\QQ ( \X_{\SL}( N ))$ is finitely generated as a $\QQ[I_\alpha]$-module.  Fortunately, the \texttt{basis} command in \textit{Macaulay2} finds generating sets for modules over specified rings \cite{M2}. 

To begin our search, we needed to know which knots in $S^3$ have large complements.  Burton, Coward, and the third author \cite{BCT} developed an algorithm to check precisely this and \cite[Appendix F]{BCT} lists all 1019 large knots in $S^3$ that can be represented with diagrams with at most 12 crossings.  

Chesebro proved that if an irreducible curve in $\X_{\SL}(N)$ ANI-detects a closed essential surface, then it detects a closed essential surface with ideal points \cite[Proposition 5.2]{C}.  Goodman, Heard, and Hodgson compiled a partial list of all hyperbolic knots and links with up to 12 crossings whose discrete and faithful representations are ANI.  Using this table, we created a list of all hyperbolic large knots with 12 or fewer crossings whose discrete faithful representation is either not ANI or its integrality is unknown.  For each knot on our list, we recorded the number of tetrahedra in the triangulation given by the program \textit{SnapPy} \cite{SnapPy}.  Using the notation from Rolfsen's table \cite{ROL}, the knots in our table with the fewest tetrahedra in their \textit{SnapPy} triangulation are $10_{152}$ with $9$ tetraheadra, $10_{153}$ with $8$, and $10_{154}$ with $10$. 

\begin{figure}[htp]

\centering
\includegraphics[width=.3\textwidth]{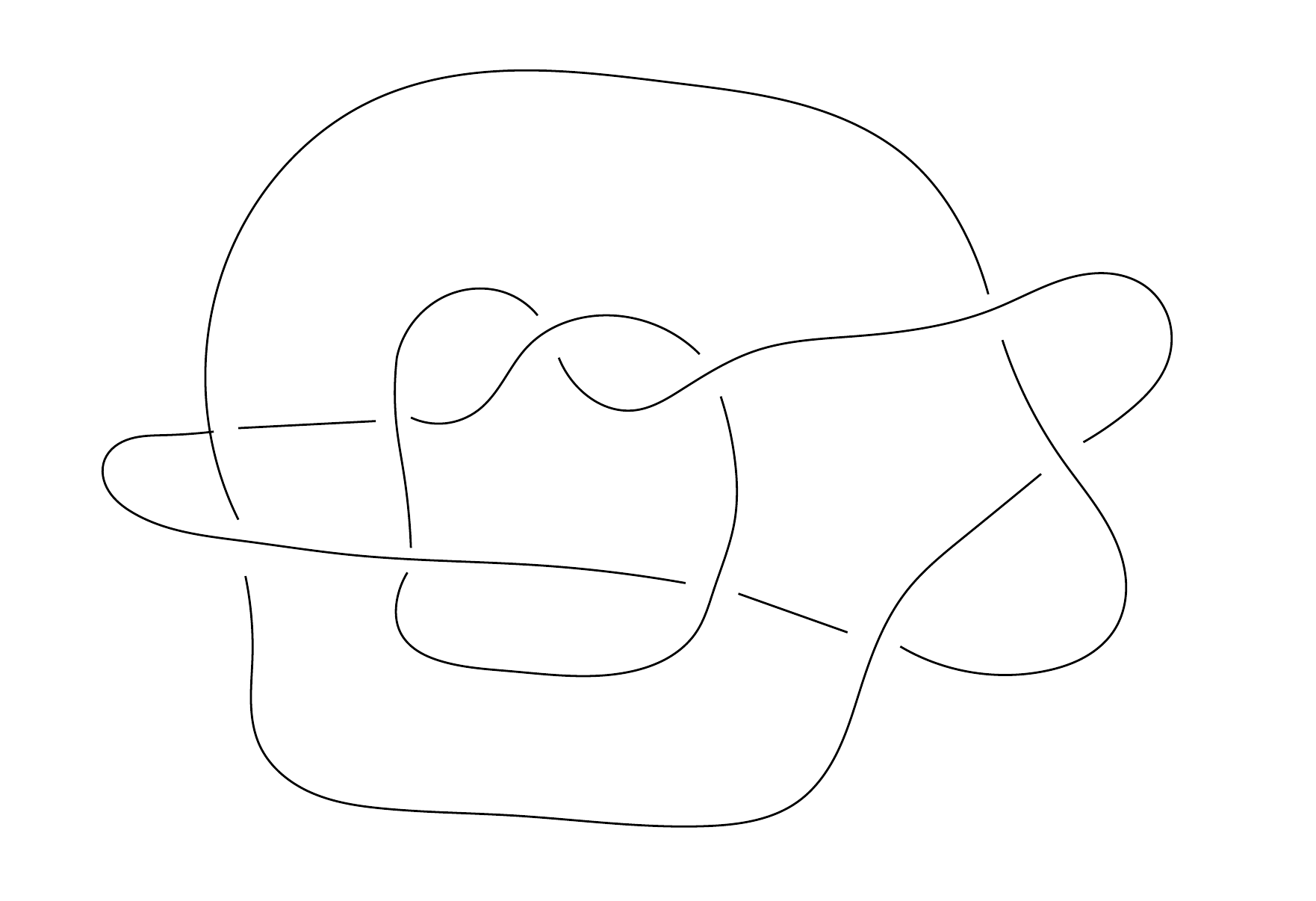}\hfill
\includegraphics[width=.3\textwidth]{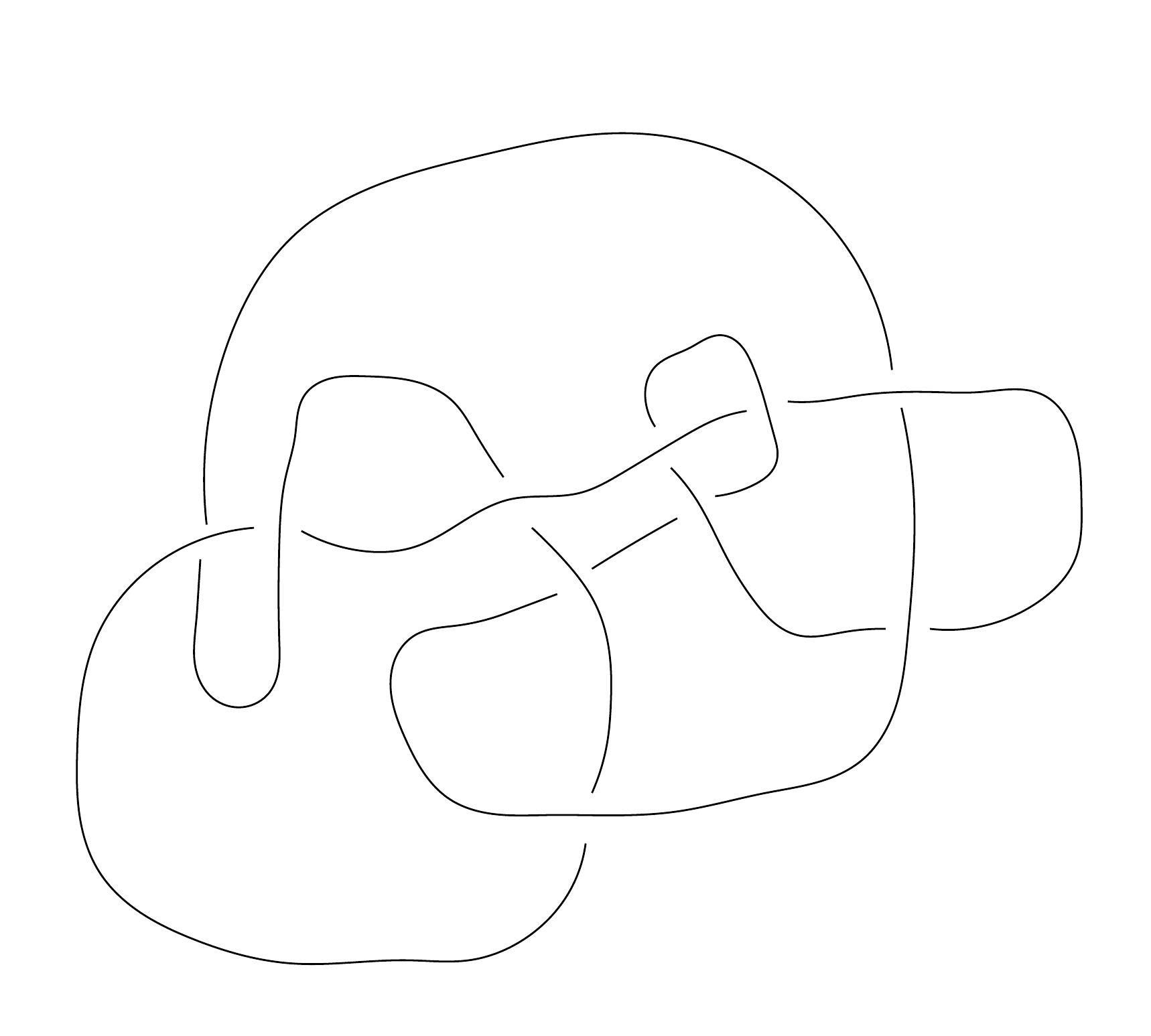}\hfill
\includegraphics[width=.3\textwidth]{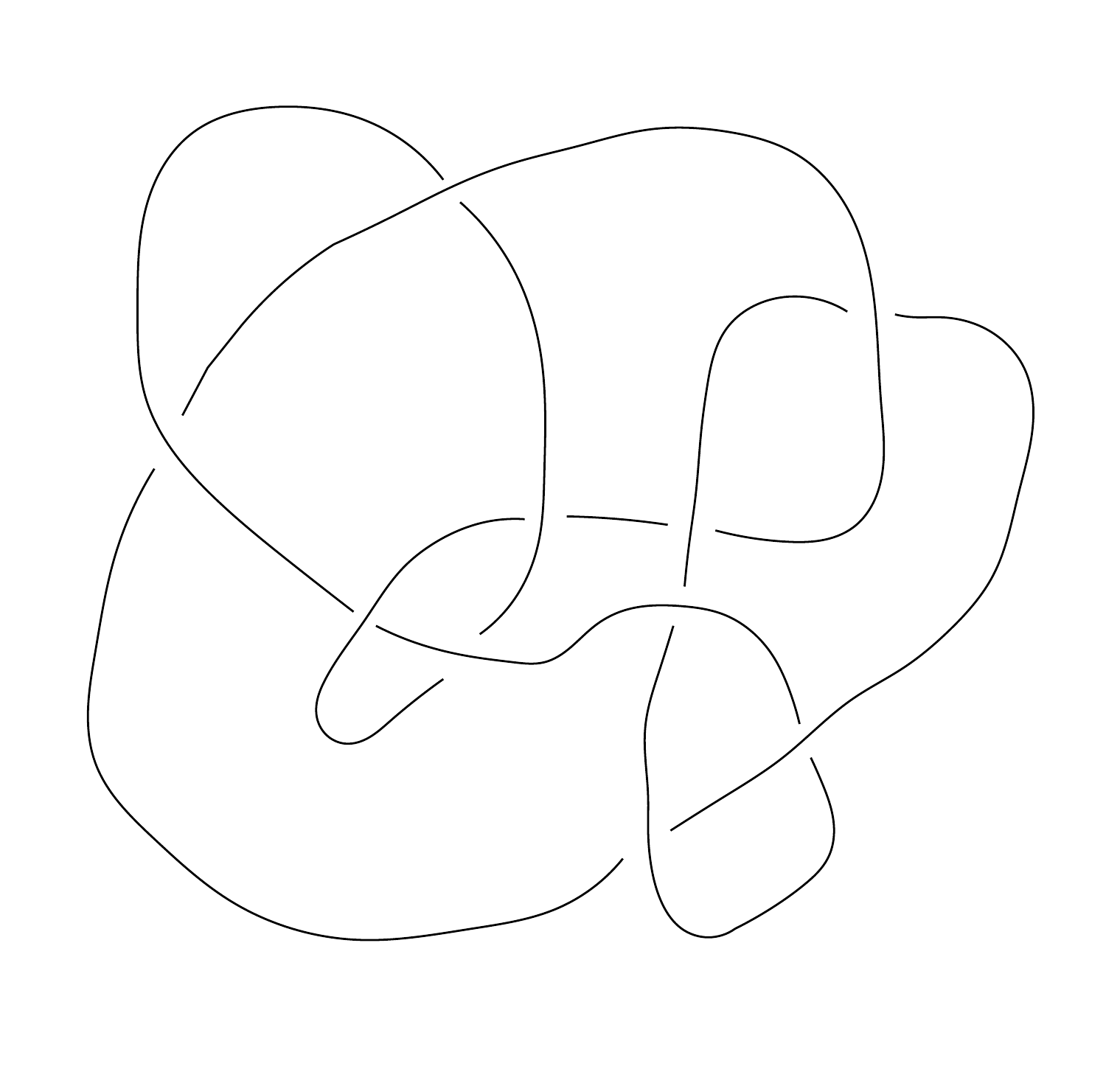}

\caption{The knots $10_{152}$, $10_{153}$, and $10_{154}$. }
\label{fig:knots}

\end{figure} 


\begin{proof}[Proof of Theorem~\ref{thm:main}]
We need to show that no essential surfaces is detected by an ideal point of the character variety.
We give the details for $N$ the complement of the knot $10_{153}$; the calculations for $10_{152}$ and $10_{154}$ follow along the same lines.
The package \textit{HIKMOT}~\cite{HIKMOT} certifies that the interior of $N$ admits a finite volume hyperbolic metric.
Now \textit{SnapPy}~\cite{SnapPy} gives the following presentation for the fundamental group of $N$:
\[  \pi_1 ( N ) \cong \langle a, b, c\mid abAbCaabAbcB, abCBcAc \rangle  \] 
where capital letters denote inverses.  A basis for $\pi_1  ( \bound N )$ is
\[ \{\; \mu = BAABa, \; \lambda = BAACaabCAbCBa \;\}. \] 
Note that $\inv{\mu} \lambda$ is conjugate to the element $CaabCAbC$ in $\pi_1(N)$, so $I_{\inv{\mu} \lambda} = I_{CaabCAbC}$. 

Now we use Propositions~\ref{prop:tgfg} and \ref{prop:dfneqns} to compute defining equations for $\X_{\SL} (N)$.  Using coordinates 
\[ ( x = I_a , y = I_b, z = I_c, w = I_{ab}, t = I_{ac}, u = I_{bc}, v = I_{abc} ) \]
we find that $\X_{\SL}(N)$ is cut out of $\CC^7$ by 9 polynomials $\{p_0, \ldots, p_8\}$: $p_0$ being the polynomial defining the character variety of the free group on 3 letters and the other 8 coming from our presentation for $\pi_1 ( N)$.  This collection of polynomials is long and unwieldy, so we elect not to display these polynomials here; the interested reader can find the polynomials in the ancillary files~\cite{CKT-anc}.

Define a new ideal $I$ generated by $\{p_0, \ldots, p_8, s - I_{\inv{\mu}\lambda}(x,y,z,w,t,u,v) \}$.  Then the zero set of $I$ is an embedding of $\X_{\SL}(N)$ into $\CC^8$ with coordinates $(x,y,z,w,t,u,v,s)$ such that $I_{\inv{\mu}\lambda}$ is the coordinate function $s$. 

Set $S = \QQ[s]$ and $ R = S[x,y,z,w,t,u,v]$.  Then $R/I \cong T_{\QQ} ( \X_{\SL} ( N ))$ since $\X_{\SL}(N)$ is defined over $\QQ$. To investigate whether or not $T_{\QQ} ( \X_{\SL} (N) )$ is finitely generated as an $S$-module, we first compute a Gr\"obner basis for $I$, then execute the command 
\[ \text{ \texttt{ basis( R / I , SourceRing => S ) }} \]
in \textit{Macaulay2}.  The command gives either an error if the module is not finitely generated over $S$ or a list of generators.  In our case, we get a list $\mathcal{L}$ of 48 monomials which can be found in the ancillary files~\cite{CKT-anc}.

While $\mathcal{L}$ may not be a free basis for $T_\QQ( \X_{\SL} (N))$ over $\QQ[I_{\inv{\mu} \lambda}]$, it is guaranteed to be at least a generating set.  In particular, by Corollary \ref{cor:overqq}, $\inv{\mu} \lambda$ is not a boundary slope of $N$ and $\X_{\SL}(N)$ does not detect any closed essential surfaces even though $N$ is a large knot manifold.
\end{proof} 

\begin{rem*}
The defining equations for $\X_{\SL}(N)$ were computed using a \textit{Mathematica} notebook written by Ashley, Burelle, and Lawton \cite{ABL} that is based on the Free Group Toolbox Version 2.0 \textit{Mathematica} notebook written by William Goldman \cite{GOLD}.  We independently verified the equations with a computation from first principles.
\end{rem*}

We first performed the above calculation for the knots $10_{152}$ and $10_{153}$ on a 2014 MacBook Air with a 1.4GHz processor and only 8GB of memory. The calculation for $10_{154}$ took 62 hours on a cluster with a 2.6GHz processor and used 33GB of memory at TU Berlin. This was done by Robert Loewe and Benjamin Lorenz. Both machines used version 1.11 of \textit{Macaulay2}. 

%

We end this note with an application of Theorem \ref{thm:main} answering Question 9 of Schanuel and Zhang~\cite{SZ} concerning algebraic non-integral representations into $\PSL_2 ( \CC)$.  Even though the trace of an element of $\PSL_2 ( \CC)$ is not well-defined, it is up to sign.  An element of a number field $F$ is integral over $\ZZ$ if and only if its negative is, so we say a representation $\rho \co \Gamma \to \PSL_2 (F)$ is \emph{algebraic non-integral} if there is some $\gamma \in \Gamma$ such that the trace of a lift of $\rho ( \gamma)$ to $\SL_2 (\CC)$ is algebraic non-integral.

\begin{cor}
There are infinitely many large closed hyperbolic 3--manifolds with no ANI-representations into $\PSL_2 ( \CC)$.  

\begin{proof} 
Let $N$ denote the complement of any knot with the property that the character variety of its complement has no ideal point detecting a closed essential surface (such as the knot $10_{152}$, $10_{153}$ or $10_{154}$).  Fix a basis $\{ \mu, \lambda \}$  for $\pi_1( \bound N) \subseteq \pi_1 ( N)$ such that $\mu$ is a meridian for $N$ and $\lambda$ may be represented by a simple closed curve on $\bound N$.   Let $N(p/q)$ denote the closed 3--manifold obtained by performing $p/q$ Dehn filling on $N$; that is, attach a solid torus to $\bound N$ such that the meridian of the solid torus is glued to $\bound N$ along a primitive curve representing  $\mu^p \lambda^q$. We obtain a presentation for $\pi_1 ( N ( p/q))$ by simply adding the relation $\mu^p \lambda^q$ to a presentation for $\pi_1(N)$.  In particular, there is an epimorphism  $\pi_1 (N ) \to \pi_1 ( N ( p/q))$ which induces inclusions 
\[ \X_{\SL} ( N(p/q)) \hookrightarrow \X_{\SL} ( N ) \quad \text{and} \quad \X_{\PSL}(N(p/q)) \to \X_{\PSL}(N). \]

There is a sequence  $\{p_i/q_i \}_{i=1}^\infty$ of slopes on $\bound N$ such that
\begin{enumerate}
\item $N(p_i / q_i)$ is  a large hyperbolic closed 3--manifold;
\item $p_i / q_i$ is not a boundary slope of $N$.
\end{enumerate} 
That $N(p_i/q_i)$ can be taken to be hyperbolic follows from Thurston's Hyperbolic Dehn Surgery Theorem \cite{THUR}.  We may assume $N(p_i/q_i)$ is large by \cite[Theorem 2.0.3]{CGLS}. Finally, $N$ has only finitely many boundary slopes \cite{HAT}, so we may choose  $p_i/q_i$ that are not boundary slopes of $N$. 

Fix $p_i / q_i$ and let $r \co \pi_1 ( N) \to \pi_1 ( N( p_i / q_i ))$ be the quotient map.   Suppose towards a contradiction that $N( p_i / q_i )$ has an algebraic non-integral representation  $\rho \co \pi_1 ( N(p_i/q_i)) \to \PSL_2 ( \CC)$.  Then $\rho \circ r$ is an ANI representation of $\pi_1(N)$ into $\PSL_2 ( \CC)$.  By \cite{CU}, this representation lifts to an ANI-representation $ \wt{\rho} \co \pi_1 (N ) \to \SL_2 ( \CC).$   

We claim that $\wt{\rho}$ must ANI-detect a closed essential surface in $N$.  Since $\rho\circ r$ factors through $\pi_1( N ( p_i / q_i))$, we must have $\wt{\rho}( \mu^{p_i}\lambda^{q_i}) = I \text{ or } -I$, where $I$ is the $2 \times 2$ identity matrix.  In particular, the character $\chi_{\wt{\rho}}$ evaluates to $\pm 2$ at    $\mu^{p_i}\lambda^{q_i}$, both of which are  integral  over  $\ZZ$.  Thus, by Theorem \ref{thm:cs1}, $\chi_{\wt{\rho}}$ either ANI-detects a closed essential surface or $\mu^{p_i}\lambda^{q_i}$ is a boundary slope of $N$.  This is not the case by construction, proving our claim.

Finally, Chesebro showed that if $\X_{\SL}(N)$ ANI-detects a closed essential surface, then it detects one by ideal points \cite[Proposition 5.2]{C}.  This contradicts our hypothesis, so $N(p_i / q_i)$ has no ANI-representations into $\PSL_2 ( \CC)$. 
\end{proof} 
\end{cor}

\bibliographystyle{plain}
\bibliography{ces}

	\address{Alex Casella,\\ School of Mathematics and Statistics F07,\\ The University of Sydney,\\ NSW 2006 Australia\\
	(casella@maths.usyd.edu.au)\\--}
	
	\address{Charles Katerba,\\ Department of Mathematical Sciences,\\ Montana State University \\ MT 59717 USA\\
	(charles.katerba@montana.edu)\\--}
	
	\address{Stephan Tillmann,\\ School of Mathematics and Statistics F07,\\ The University of Sydney,\\ NSW 2006 Australia\\
	(tillmann@maths.usyd.edu.au)}

	\Addresses


\end{document}